\documentclass{amsart}
\usepackage{amsfonts}
\usepackage{amsmath}
\usepackage{hyperref}

\newtheorem{theorem}{Theorem}[section]
\theoremstyle{plain}

\newtheorem{corollary}{Corollary}

\newtheorem{lemma}{Lemma}[section]

\newtheorem{remark}{Remark}[section]

\numberwithin{equation}{section}
\input{tcilatex}

\begin{document}
\title[Sharp bounds for the second Seiffert mean]{Sharp bounds for the
second Seiffert mean in terms of power means }
\author{Zhen-Hang Yang}
\address{System Division, Zhejiang Province Electric Power Test and Research
Institute, Hangzhou, Zhejiang, China, 31001}
\email{yzhkm@163.com}
\date{March 17, 2012}
\thanks{This paper is in final form and no version of it will be submitted
for publication elsewhere.}

\begin{abstract}
For $a,b>0$ with $a\neq b$, let $T\left( a,b\right) $ denote the second
Seiffert mean defined by 
\begin{equation*}
T\left( a,b\right) =\frac{a-b}{2\arctan \frac{a-b}{a+b}}
\end{equation*}%
and $A_{r}\left( a,b\right) $ denote the $r$-order power mean. We present
the sharp bounds for the second Seiffert mean in terms of power means: 
\begin{equation*}
A_{p_{1}}\left( a,b\right) <T\left( a,b\right) \leq A_{p_{2}}\left(
a,b\right) ,
\end{equation*}%
where $p_{1}=$ $\log _{\pi /2}2$ and $p_{2}=5/3$ can not be improved.
\end{abstract}

\maketitle

\section{Introduction}

Throughout the paper, we assume that $a,b>0$ with $a\neq b$. The power mean
of order $r$ of the positive real numbers $a$ and $b$ is defined by

\begin{equation*}
A_{r}=A_{r}(a,b)=\left( \frac{a^{r}+b^{r}}{2}\right) ^{1/r}\text{ if }r\neq 0%
\text{ and }A_{0}=A_{0}(a,b)=\sqrt{ab}.
\end{equation*}%
It is well-known that the function $r\mapsto A_{r}(a,b)$ is continuous and
strictly increasing on $\mathbb{R}$ (see \cite{Bullen.1988}). As special
cases, the arithmetic mean, geometric mean and quadratic mean are $A=A\left(
a,b\right) =A_{1}\left( a,b\right) $, $G=G\left( a,b\right) =A_{0}\left(
a,b\right) $ and $Q=Q\left( a,b\right) =A_{2}\left( a,b\right) $,
respectively.

The Lehmer mean of order $r$ of the positive real numbers $a$ and $b$ is
defined as 
\begin{equation*}
L_{r}=L_{r}\left( a,b\right) =\frac{a^{r+1}+b^{r+1}}{a^{r}+b^{r}}
\end{equation*}%
(see \cite{Lehmer.36(1971)}). It is seen that the function $r\mapsto
L_{r}(a,b)$ is continuous and strictly increasing on $\mathbb{R}$. In
particular, $L_{0}=A$, $L_{1}=C$ are the arithmetic mean, contra-harmonic
mean, respectively. Clearly, Lehmer mean can be expressed by power means as $%
L_{r}=A_{r+1}^{r+1}A_{r}^{-r}$.

The first Seiffert mean \cite{Seiffert.42(1987)} is defined by%
\begin{equation*}
P=P\left( a,b\right) =\frac{a-b}{2\arcsin \frac{a-b}{a+b}}.
\end{equation*}%
Many remarkable inequalities for $P$ can be found in the literature \cite%
{Jagers.12(1994)}, \cite{Sandor.76(2001)}, \cite{Hasto.3(5)(2002)}, \cite%
{Sandor.2003(16)(2003)}, \cite{Neuman.17(1)(2006)}, \cite{Chu.AAA.2010}, 
\cite{He.17(4)(2010)}, \cite{Wang.4(21-24)(2010)}, \cite{Wang.4(4)(2010)},%
\cite{Chu.JIA.2010}, \cite{Liu.JIA.2011}. Here we mention that the following
sharp bounds for the first Seiffert mean $P$ in terms of power means proved
by Jagers \cite{Jagers.12(1994)} and H\"{a}st\"{o} \cite{Hasto.7(1)(2004)}:%
\begin{equation}
A_{\log _{\pi }2}\left( a,b\right) <P\left( a,b\right) <A_{2/3}\left(
a,b\right) .  \label{H-J}
\end{equation}

In 1995, Seiffert \cite{Seiffert.29(1995)} defined his second mean as%
\begin{equation*}
T=T\left( a,b\right) =\frac{a-b}{2\arctan \frac{a-b}{a+b}}
\end{equation*}%
and proved that 
\begin{equation}
A<T<Q.  \label{Seiffert}
\end{equation}

S\'{a}ndor \cite[pp. 265-267]{Sandor.2009} showed that by a transformation
of arguments, the mean $T$ can be reduced to the mean $P$: 
\begin{equation*}
T(a,b)=P(x,y),
\end{equation*}%
where 
\begin{equation}
x=\frac{\sqrt{2\left( a^{2}+b^{2}\right) }+a-b}{2}\text{, \ \ \ }y=\frac{%
\sqrt{2\left( a^{2}+b^{2}\right) }-a+b}{2},  \label{Sandor transf.}
\end{equation}%
which implies%
\begin{equation*}
A\left( x,y\right) =Q\left( a,b\right) \text{, \ \ \ }G\left( x,y\right)
=A\left( a,b\right) .
\end{equation*}%
Therefore, by using the transformations (\ref{Sandor transf.}), the
following transformations of means will be true:

\begin{equation*}
G\rightarrow A\text{, \ \ \ }A\rightarrow Q\text{, \ \ \ }P\rightarrow T.
\end{equation*}%
Thus, from the known inequalities involving $P$, $A$, $G$ he easily obtained
corresponding ones involving $T$, $Q$, $A$, for example, (\ref{Seiffert})
and the following inequalities:%
\begin{equation}
Q^{2/3}A^{1/3}<Q^{1/3}\left( \frac{Q+A}{2}\right) ^{2/3}<T<\frac{2Q+A}{3}.
\label{Sandor}
\end{equation}

Recently, Chu et al. in \cite{Chu.JIA.2011} proved the double inequality 
\begin{equation}
p_{1}Q+\left( 1-p_{1}\right) A<T<q_{1}Q+\left( 1-q_{1}\right) A  \label{Chu1}
\end{equation}%
holds if and only if $p_{1}\leq \left( \sqrt{2}+1\right) \left( 4-\pi
\right) /\pi ,q_{1}\geq 2/3$, which shows that the constant $2/3$ of the
third inequality in (\ref{Sandor}) is the best.

Very recently, Witkowski \cite{Witkowski.MIA.2012.inprint} used some
geometric ideas to prove a series of inequalities involving $T$, $Q$, $A$,
such as 
\begin{eqnarray}
A &<&T<\frac{4}{\pi }A,  \label{W1} \\
\frac{2\sqrt{2}}{\pi }Q &<&T<Q,  \label{W2} \\
(1-r_{1})Q+r_{1}A &<&T<\frac{2Q+A}{3},  \label{W3}
\end{eqnarray}%
where $r_{1}=\frac{2\left( \pi -2\sqrt{2}\right) }{\left( 2-\sqrt{2}\right)
\pi }=\allowbreak 0.340341385...$. It is obvious that (\ref{W3}) is actually
(\ref{Chu1}).

In 2010, Wang et al. \cite{Wang.4(4)(2010)} presented the optimal upper and
lower Lehmer mean bounds for $T$ as follows: 
\begin{equation}
L_{0}<T<L_{1/3}.  \label{Wang}
\end{equation}

In \cite{Chu.AAA.2012.inprint}, Chu et al. demonstrated that the double
inequality%
\begin{equation}
C\left( p_{2}a+\left( 1-p_{2}\right) b,p_{2}b+\left( 1-p_{2}\right) a\right)
<T\left( a,b\right) <C\left( q_{2}a+\left( 1-q_{2}\right) b,q_{2}b+\left(
1-q_{2}\right) a\right)   \label{Chu2}
\end{equation}%
if and only if $p_{2}\leq \left( 1+\sqrt{4/\pi -1}\right) /2$, $q_{2}\geq
\left( 3+\sqrt{3}\right) /6$.

It is interesting and useful to evaluate the second Seiffert mean $T$ by
power means $A_{p}$. Until recently, the inequalities (\ref{Seiffert}) has
improved by Constin and Toader \cite{Costin.IJMMS.2012.inprint} as 
\begin{equation}
N<A_{3/2}<T<Q,  \label{C-S}
\end{equation}%
where $N$ is the Neuman-S\'{a}ndor mean defined in \cite{Neuman.17(1)(2006)}
by%
\begin{equation*}
N=N\left( a,b\right) =\frac{a-b}{2\func{arcsinh}\frac{a-b}{a+b}}.
\end{equation*}%
Up to now, this may be the best result for the bounds for the second
Seiffert mean in terms of power means. For this reason, we are going to find
the best $p\in \left( 3/2,2\right) $ such that the inequality%
\begin{equation}
T\left( a,b\right) <A_{p}\left( a,b\right)   \label{m}
\end{equation}%
or its reverse inequality holds in this paper.

Our main results are the following

\begin{theorem}
\label{Theorem 1}The inequality (\ref{m}) if and only if $p\geq p_{2}=5/3$.
Moreover, we have%
\begin{equation}
\alpha _{1}A_{5/3}\left( a,b\right) <T\left( a,b\right) <\alpha
_{2}A_{5/3}\left( a,b\right) ,  \label{ma}
\end{equation}%
where $\alpha _{1}=2^{8/5}\pi ^{-1}=\allowbreak 0.964\,94...$ and $\alpha
_{2}=1$ are the best possible constants.
\end{theorem}

\begin{theorem}
\label{Theorem 2}The inequality (\ref{m}) is reversed if and only if $p\leq
p_{1}=\log _{\pi /2}2=\allowbreak 1.\,\allowbreak 534\,9...$. Moreover, we
have%
\begin{equation}
\beta _{1}A_{\log _{\pi /2}2}\left( a,b\right) <T\left( a,b\right) <\beta
_{2}A_{\log _{\pi /2}2}\left( a,b\right) ,  \label{mb}
\end{equation}%
where $\beta _{1}=1$ and $\beta _{2}=1\allowbreak .\,\allowbreak 013\,6...$
are the best possible constants.
\end{theorem}

\section{Lemmas}

In order to prove our main results, we need the following lemmas.

\begin{lemma}
\label{Lemma 2.1}Let $F_{p}$ be the function defined on $\left( 0,1\right) $
by%
\begin{equation}
F_{p}\left( x\right) =\ln \frac{T\left( 1,x\right) }{A_{p}\left( 1,x\right) }%
=\ln \frac{1-x}{2\arctan \frac{1-x}{x+1}}-\frac{1}{p}\ln \left( \frac{x^{p}+1%
}{2}\right) .  \label{f}
\end{equation}%
Then we have 
\begin{eqnarray}
\lim_{x\rightarrow 1^{-}}\frac{F_{p}\left( x\right) }{\left( x-1\right) ^{2}}
&=&-\frac{1}{24}\left( 3p-5\right) ,  \label{2.1} \\
F_{p}\left( 0^{+}\right)  &=&\lim_{x\rightarrow 0^{+}}F_{p}\left( x\right)
=\left\{ 
\begin{array}{lc}
\frac{1}{p}\ln 2-\ln \frac{\pi }{2} & \text{if }p>0, \\ 
\infty  & \text{if }p\leq 0,%
\end{array}%
\right.   \label{2.2}
\end{eqnarray}%
where $F_{0}\left( x\right) :=\lim_{p\rightarrow 0}F_{p}\left( x\right) $.
\end{lemma}

\begin{proof}
Using power series expansion we have 
\begin{equation*}
F_{p}\left( x\right) =\allowbreak -\frac{1}{24}\left( 3p-5\right) \left(
x-1\right) ^{2}+O\left( \left( x-1\right) ^{3}\right) ,
\end{equation*}%
which yields (\ref{2.1}).

Direct limit calculation leads to (\ref{2.2}), which proves the lemma.
\end{proof}

\begin{lemma}
\label{Lemma 2.2}Let $F_{p}$ be the function defined on $\left( 0,1\right) $
by (\ref{f}). Then $F_{p}$ is strictly increasing on $\left( 0,1\right) $ if
and only if $p\geq 5/3$ and decreasing on $\left( 0,1\right) $\ if and only
if $p\leq 1$.
\end{lemma}

\begin{proof}
Differentiation yields%
\begin{equation}
F_{p}^{\prime }\left( x\right) =\frac{x^{p-1}+1}{x\left( 1-x\right) \left(
x^{p}+1\right) \arctan \frac{1-x}{x+1}}f_{1}\left( x\right) ,  \label{2.3}
\end{equation}%
where 
\begin{equation}
f_{1}\left( x\right) =\frac{\allowbreak \left( 1-x\right) \left(
x^{p}+1\right) }{\allowbreak \left( x^{2}+1\right) \left( x^{p-1}+1\right) }%
-\arctan \frac{1-x}{x+1}.  \label{2.4}
\end{equation}%
Differentiation again leads to%
\begin{equation}
f_{1}^{\prime }\left( x\right) =-\allowbreak \frac{x\left( 1-x\right) }{%
\left( x^{2}+1\right) ^{2}\left( x^{p-1}+1\right) ^{2}}f_{2}\left( x\right) ,
\label{2.5}
\end{equation}%
where%
\begin{equation}
f_{2}\left( x\right) =\left( \left( 1-p\right) x^{p}+\left( p+1\right)
x^{p-1}-2x^{2p-3}-\left( p+1\right) x^{p-2}+\left( p-1\right)
x^{p-3}+2\right) .  \label{2.6}
\end{equation}

(i) We now prove that $F_{p}$ is strictly increasing on $\left( 0,1\right) $
if and only if $p\geq 5/3$. From (\ref{2.3}) it is seen that $\func{sgn}%
F_{p}^{\prime }\left( x\right) =\func{sgn}$ $f_{1}\left( x\right) $ for $%
x\in \left( 0,1\right) $, so it suffices to prove that $f_{1}\left( x\right)
>0$ for $x\in \left( 0,1\right) $ if and only if $p\geq 5/3$.

\textbf{Necessity}. If $f_{1}\left( x\right) >0$ for $x\in \left( 0,1\right) 
$ then there must be $\lim_{x\rightarrow 1^{-}}\left( 1-x\right)
^{-3}f_{1}\left( x\right) \geq 0$. Application of L'Hospital rule leads to 
\begin{equation*}
\lim_{x\rightarrow 1^{-}}\frac{f_{1}\left( x\right) }{\left( 1-x\right) ^{3}}%
=\lim_{x\rightarrow 1^{-}}\frac{\frac{\allowbreak \left( 1-x\right) \left(
x^{p}+1\right) }{\allowbreak \left( x^{2}+1\right) \left( x^{p-1}+1\right) }%
-\arctan \frac{1-x}{x+1}}{\left( 1-x\right) ^{3}}=\frac{1}{24}\left(
3p-5\right) ,
\end{equation*}%
and so we have $p\geq 5/3$.

\textbf{Sufficiency}. We now prove $f_{1}\left( x\right) >0$ for $x\in
\left( 0,1\right) $ if $p\geq 5/3$. As mentioned previous, the function 
\begin{equation*}
p\mapsto L_{p-1}\left( 1,x\right) =\frac{\allowbreak x^{p}+1}{x^{p-1}+1}
\end{equation*}%
is increasing on $\mathbb{R}$, it is enough to show that $f_{1}\left(
x\right) >0$ for $x\in \left( 0,1\right) $ when $p=5/3$. In this case, we
have 
\begin{equation*}
3x^{4/3}f_{2}\left( x\right) =-2x^{3}+8x^{2}-6x^{5/3}+6x^{4/3}-8x+2.
\end{equation*}%
Factoring yields 
\begin{equation*}
3x^{4/3}f_{2}\left( x\right) =2\left( 1-\sqrt[3]{x}\right) ^{3}\left(
x^{2/3}+1\right) \left( x^{4/3}+3x+5x^{2/3}+3x^{1/3}+1\right) >0.\allowbreak 
\end{equation*}%
It follows from (\ref{2.5}) that $f_{1}^{\prime }\left( x\right) <0$, that
is, the function $f_{1}$ is decreasing on $\left( 0,1\right) $. Hence for $%
x\in \left( 0,1\right) $ we have $f_{1}\left( x\right) >f_{1}\left( 1\right)
=0$, which proves the sufficiency.

(ii) We next prove that $F_{p}$ is strictly decreasing on $\left( 0,1\right) 
$ if and only if $p\leq 1$. Similarly, it suffices to show that $f_{1}\left(
x\right) <0$ for $x\in \left( 0,1\right) $ if and only if $p\leq 1$.

\textbf{Necessity}. If $f_{1}\left( x\right) <0$ for $x\in \left( 0,1\right) 
$ then we have 
\begin{equation*}
\lim_{x\rightarrow 0^{+}}f_{1}\left( x\right) =\left\{ 
\begin{array}{ll}
1-\frac{\pi }{4}>0 & \text{if }p>1, \\ 
\frac{1}{2}-\frac{\pi }{4}<0 & \text{if }p=1, \\ 
-\frac{\pi }{4} & \text{if }p<1%
\end{array}%
\leq 0,\right. 
\end{equation*}%
which yields $p\leq 1$.

\textbf{Sufficiency}. We prove $f_{1}\left( x\right) <0$ for $x\in \left(
0,1\right) $ if $p\leq 1$. Due to the monotonicity of the function $p\mapsto
L_{p-1}\left( 1,x\right) $, it suffices to demonstrate $f_{1}\left( x\right)
<0$ for $x\in \left( 0,1\right) $ when $p=1$. In this case, we have $%
f_{2}\left( x\right) =\allowbreak 4-4x^{-1}<0$, then $f_{1}^{\prime }\left(
x\right) >0$, and then for $x\in \left( 0,1\right) $ we have $f_{1}\left(
x\right) <f_{1}\left( 1\right) =0$, which proves the sufficiency and the
proof of this lemma is finished.
\end{proof}

\begin{lemma}
\label{Lemma 2.3}Let $f_{3}$ be the function defined on $\left( 0,1\right) $
by 
\begin{eqnarray}
f_{3}\left( x\right)  &=&-p\left( p-1\right) x^{3}+\left( p-1\right) \left(
p+1\right) x^{2}-2\left( 2p-3\right) x^{p}  \label{2.7} \\
&&-\left( p+1\right) \left( p-2\right) x+\left( p-1\right) \left( p-3\right) 
\notag
\end{eqnarray}%
Then $f_{3}$ is strictly increasing on $\left( 0,1\right) $ if $p\in \left(
1,5/3\right) $.
\end{lemma}

\begin{proof}
Differentiation yields$\allowbreak $%
\begin{equation}
f_{3}^{\prime }\left( x\right) =-3p\left( p-1\right) x^{2}+2\left(
p-1\right) \left( p+1\right) x-2p\left( 2p-3\right) x^{p-1}-\left(
p+1\right) \left( p-2\right) .  \label{2.8}
\end{equation}%
Note that $1<p<5/3$, using basic inequality for means%
\begin{equation*}
x^{p-1}\leq \left( p-1\right) x+\left( 2-p\right) \text{ \ }(x>0)
\end{equation*}%
to the last member of the third term in (\ref{2.8}) we have 
\begin{eqnarray*}
f_{3}^{\prime }\left( x\right)  &\geq &-3p\left( p-1\right) x^{2}+2\left(
p-1\right) \left( p+1\right) x \\
&&-2p\left( 2p-3\right) \left( \left( p-1\right) x+\left( 2-p\right) \right)
-\left( p+1\right) \left( p-2\right)  \\
&=&-3p\left( p-1\right) x^{2}\allowbreak -2\left( p-1\right) \left(
2p^{2}-4p-1\right) x+\allowbreak \left( p-2\right) \left( 4p^{2}-7p-1\right) 
\\
&:&=f_{4}\left( x\right) .
\end{eqnarray*}%
Thus, in order to prove $f_{3}^{\prime }\left( x\right) >0$, it needs to
show that $f_{4}\left( x\right) >0$ for $x\in \left( 0,1\right) $.

Since $f_{4}^{\prime \prime }\left( x\right) =-6p\left( p-1\right) <0$ and
for $p\in \left( 1,5/3\right) $ 
\begin{eqnarray*}
f_{4}\left( 0^{+}\right)  &=&\left( p-2\right) \left( p-\tfrac{\sqrt{65}+7}{8%
}\right) \left( p+\tfrac{\sqrt{65}-7}{8}\right) >0, \\
f_{4}\left( 1\right)  &=&6p\left( \frac{5}{3}-p\right) >0,
\end{eqnarray*}%
application of properties of concave functions yields for $x\in \left(
0,1\right) $ 
\begin{equation*}
f_{4}\left( x\right) >\left( 1-x\right) f_{4}\left( 0^{+}\right)
+xf_{4}\left( 1\right) >0,
\end{equation*}%
which completes the proof.
\end{proof}

\begin{lemma}
\label{Lemma 2.4}Let $p\in \left( 1,5/3\right) $ and let the function $%
x\mapsto F_{p}\left( x\right) $ be defined on $\left( 0,1\right) $ by (\ref%
{f}). Then the equation $f_{1}\left( x\right) =0$ has a unique solution $%
x_{3}$ such that $F_{p}$ is increasing on $\left( 0,x_{3}\right) $ and
decreasing on $\left( x_{3},1\right) $, where $f_{1}\left( x\right) $ is
defined by (\ref{2.4}).
\end{lemma}

\begin{proof}
Differentiating $f_{2}\left( x\right) $ defined by (\ref{2.6}) gives%
\begin{equation}
x^{4-p}f_{2}^{\prime }\left( x\right) =f_{3}\left( x\right) ,  \label{2.9}
\end{equation}%
where $\allowbreak f_{3}\left( x\right) $ is defined by (\ref{2.7}).

Because that $f_{3}$ is strictly increasing on $\left( 0,1\right) $ if $p\in
\left( 1,5/3\right) $ by Lemma (\ref{Lemma 2.3}) and note that 
\begin{equation*}
f_{3}\left( 0^{+}\right) =\allowbreak \left( p-1\right) \left( p-3\right) <0%
\text{, \ \ \ }f_{3}\left( 1\right) =2\allowbreak \left( 5-3p\right) >0,
\end{equation*}%
there is a unique $x_{1}\in $ $\left( 0,1\right) $ such that $f_{3}\left(
x\right) <0$ for $x\in \left( 0,x_{1}\right) $ and $f_{3}\left( x\right) >0$
for $x\in \left( x_{1},1\right) $. Then it is seen from (\ref{2.9}) that $%
f_{2}$ is decreasing on $\left( 0,x_{1}\right) $ and increasing on $\left(
x_{1},1\right) $, which yields $f_{2}\left( x\right) <f_{2}\left( 1\right) =0
$ for $x\in \left( x_{1},1\right) $. This together with $\limfunc{sgn}%
f_{2}\left( 0^{+}\right) =\limfunc{sgn}\left( p-1\right) >0$ reveals that
there exits a unique $x_{2}\in \left( 0,x_{1}\right) $ such that $%
f_{2}\left( x\right) >0$ for $x\in \left( 0,x_{2}\right) $ and $f_{2}\left(
x\right) <0$ for $x\in \left( x_{2},1\right) $. It follows from (\ref{2.5})
that $f_{1}$ is decreasing on $\left( 0,x_{2}\right) $ and increasing on $%
\left( x_{2},1\right) $, and therefore $f_{1}\left( x\right) <f_{1}\left(
1\right) =0$ for $x\in \left( x_{2},1\right) $, which in combination with $%
f_{1}\left( 0^{+}\right) =\allowbreak 1-\frac{1}{4}\pi >0$ indicates that
there is a unique $x_{3}\in \left( 0,x_{2}\right) $ such that $f_{1}\left(
x\right) >0$ for $x\in \left( 0,x_{3}\right) $ and $f_{1}\left( x\right) <0$
for $x\in \left( x_{3},1\right) $. By (\ref{2.3}) it is easy to see that the
function $x\mapsto F_{p}\left( x\right) $ is increasing on $\left(
0,x_{3}\right) $ and decreasing on $\left( x_{3},1\right) $, which proves
the lemma.
\end{proof}

\section{Proofs of Main Results}

Based on the lemmas in the above section, we can easily proved our main
results.

\begin{proof}[Proof of Theorem \protect\ref{Theorem 1}]
By symmetry, we assume that $a>b>0$. Then inequality (\ref{m}) is equivalent
to%
\begin{equation}
\ln T\left( 1,x\right) -\ln A_{p}\left( 1,x\right) =F_{p}\left( x\right) <0,
\label{3.1}
\end{equation}%
where $x=b/a\in \left( 0,1\right) $. Now we prove the inequality (\ref{3.1})
holds for all $x\in \left( 0,1\right) $ if and only if $p\geq 5/3$. 

\textbf{Necessity}. If inequality (\ref{3.1}) holds, then by Lemma \ref%
{Lemma 2.1} we have%
\begin{equation*}
\left\{ 
\begin{array}{l}
\lim_{x\rightarrow 1^{-}}\frac{F_{p}\left( x\right) }{\left( x-1\right) ^{2}}%
=-\frac{1}{24}\left( 3p-5\right) \leq 0, \\ 
\lim_{x\rightarrow 0^{+}}F_{p}\left( x\right) =\frac{1}{p}\ln 2-\ln \frac{%
\pi }{2}\leq 0\text{ if }p>0,%
\end{array}%
\right.
\end{equation*}%
which yields $p\geq 5/3$.

\textbf{Sufficiency. }Suppose that $p\geq 5/3$. It follows from Lemma \ref%
{Lemma 2.2} that $F_{p}\left( x\right) <F_{p}\left( 1\right) =0$ for $x\in
\left( 0,1\right) $, which proves the sufficiency.

Using the monotonicity of the function $x\mapsto F_{5/3}\left( x\right) $ on 
$\left( 0,1\right) $, we have 
\begin{equation*}
\ln \left( 2^{8/5}\pi ^{-1}\right) =F_{5/3}\left( 0^{+}\right)
<F_{5/3}\left( x\right) <F_{5/3}\left( 1^{-}\right) =0,
\end{equation*}%
which implies (\ref{ma}).

Thus the proof of Theorem \ref{Theorem 1} is finished.
\end{proof}

\begin{proof}[Proof of Theorem \protect\ref{Theorem 2}]
Clearly, the reverse inequality of (\ref{m}) is equivalent to%
\begin{equation}
\ln T\left( 1,x\right) -\ln A_{p}\left( 1,x\right) =F_{p}\left( x\right) >0,
\label{3.2}
\end{equation}%
where $x=b/a\in \left( 0,1\right) $. Now we show that the inequality (\ref%
{3.2}) holds for all $x\in \left( 0,1\right) $ if and only if $p\leq \log
_{\pi /2}2$.

\textbf{Necessity}. The condition $p\leq \log _{\pi /2}2$ is necessary.
Indeed, if inequality (\ref{3.2}) holds, then we have%
\begin{equation*}
\left\{ 
\begin{array}{l}
\lim_{x\rightarrow 1^{-}}\frac{F_{p}\left( x\right) }{\left( x-1\right) ^{2}}%
=-\frac{1}{24}\left( 3p-5\right) \geq 0, \\ 
\lim_{x\rightarrow 0^{+}}F_{p}\left( x\right) =\frac{1}{p}\ln 2-\ln \frac{%
\pi }{2}\geq 0\text{ if }p>0%
\end{array}%
\right.
\end{equation*}%
or%
\begin{equation*}
\left\{ 
\begin{array}{l}
\lim_{x\rightarrow 1^{-}}\frac{F_{p}\left( x\right) }{\left( x-1\right) ^{2}}%
=-\frac{1}{24}\left( 3p-5\right) \geq 0, \\ 
\lim_{x\rightarrow 0^{+}}F_{p}\left( x\right) =\infty \text{ if }p\leq 0.%
\end{array}%
\right.
\end{equation*}%
Solving the above inequalities leads to $p\leq \log _{\pi /2}2$.

\textbf{Sufficiency. }The condition $p\leq \log _{\pi /2}2$ is also
sufficient. Since the function $r\mapsto A_{r}\left( 1,x\right) $ is
increasing, so the function $p\mapsto F_{p}\left( x\right) $ is decreasing,
thus it is suffices to show that $F_{p}\left( x\right) >0$ for all $x\in
\left( 0,1\right) $ if $p=p_{1}=\log _{\pi /2}2$.

Lemma \ref{Lemma 2.4} reveals that for $p\in \left( 1,5/3\right) $ there is
a unique $x_{3}$ to satisfy 
\begin{equation}
f_{1}\left( x_{3}\right) =\frac{\allowbreak \left( 1-x_{3}\right) \left(
x_{3}^{p}+1\right) }{\allowbreak \left( x_{3}^{2}+1\right) \left(
x_{3}^{p-1}+1\right) }-\arctan \frac{1-x_{3}}{x_{3}+1}=0  \label{3.3}
\end{equation}%
such that the function $x\mapsto F_{p}\left( x\right) $ is strictly
increasing on $\left( 0,x_{3}\right) $ and strictly decreasing on $\left(
x_{3},1\right) $. It is acquired that for $p_{1}=\log _{\pi /2}2\in \left(
1,5/3\right) $ 
\begin{eqnarray*}
0 &=&F_{p_{1}}\left( 0^{+}\right) <F_{p_{1}}\left( x\right) \leq
F_{p_{1}}\left( x_{3}\right)  \\
0 &=&F_{p_{1}}\left( 1\right) <F_{p_{1}}\left( x_{3}\right) \leq
F_{p_{1}}\left( x_{3}\right) ,
\end{eqnarray*}%
which leads to%
\begin{equation*}
A_{p_{1}}\left( 1,x\right) <T\left( 1,x\right) <\left( \exp F_{p}\left(
x_{3}\right) \right) A_{p_{1}}\left( 1,x\right) .
\end{equation*}%
Solving the equation (\ref{3.3}) for $x_{3}$ by mathematical computation
software we find that $x_{3}\in \left(
0.186930110570624,0.186930110570625\right) $, and then 
\begin{equation*}
\beta _{2}=\exp \left( F_{p_{1}}\left( x_{3}\right) \right) \approx
\allowbreak 1.\,\allowbreak 013\,6,
\end{equation*}%
which proves the sufficiency and inequalities of (\ref{mb}).
\end{proof}

\section{Remarks}

\begin{remark}
From the proof of Lemma \ref{Lemma 2.2}, it is seen that $f_{1}\left(
x\right) >0$ if and only if $p\geq 5/3$, which implies that the inequality 
\begin{equation*}
T\left( 1,x\right) =\frac{x-1}{2\arctan \frac{x-1}{x+1}}>\frac{\left(
x^{2}+1\right) \left( x^{p-1}+1\right) }{2\left( x^{p}+1\right) }
\end{equation*}%
holds if and only $p\geq 5/3$. In a similar way, the inequality 
\begin{equation*}
T\left( 1,x\right) <\frac{\left( x^{2}+1\right) \left( x^{p-1}+1\right) }{%
2\left( x^{p}+1\right) }
\end{equation*}%
is valid if and only if $p\leq 1$. The results can be restated as a
corollary.
\end{remark}

\begin{corollary}
The inequalities%
\begin{equation}
\frac{\left( a^{2}+b^{2}\right) \left( a^{2/3}+b^{2/3}\right) }{2\left(
a^{5/3}+b^{5/3}\right) }<T\left( a,b\right) <\frac{a^{2}+b^{2}}{a+b}
\label{4.0}
\end{equation}%
with the best constants $5/3$ and $1$, and the function 
\begin{equation*}
p\mapsto \frac{\left( a^{2}+b^{2}\right) \left( a^{p-1}+b^{p-1}\right) }{%
2\left( a^{p}+b^{p}\right) }
\end{equation*}%
is decreasing. 

In particular, putting $p=1$, $1/2$, ... ,$\rightarrow -\infty $ and $5/3$, $%
2$, ..., $\rightarrow \infty $ we get 
\begin{eqnarray*}
\frac{a^{2}+b^{2}}{2\max \left( a,b\right) } &<&\cdot \cdot \cdot <\frac{a+b%
}{2}<\frac{\left( a^{2}+b^{2}\right) \left( a^{2/3}+b^{2/3}\right) }{2\left(
a^{5/3}+b^{5/3}\right) }<T\left( a,b\right)  \\
&<&\frac{a^{2}+b^{2}}{a+b}<\frac{a^{2}+b^{2}}{2\sqrt{ab}}<\cdot \cdot \cdot <%
\frac{a^{2}+b^{2}}{2\min \left( a,b\right) }.
\end{eqnarray*}
\end{corollary}

\begin{remark}
Using the monotonicity of the function defined on $\left( 0,1\right) $ by 
\begin{equation*}
F_{p}\left( x\right) =\ln \frac{T\left( 1,x\right) }{A_{p}\left( 1,x\right) }
\end{equation*}%
given in Lemma \ref{Lemma 2.2}, we can obtain a Fan Ky type inequality but
omit the further details of the proof.

\begin{corollary}
Let $a_{1},a_{2},b_{1},b_{2}>0$ with $a_{1}/b_{1}<a_{2}/b_{2}<1$. Then the
following Fan Ky type inequality%
\begin{equation*}
\frac{T\left( a_{1},b_{1}\right) }{T\left( a_{2},b_{2}\right) }<\frac{%
A_{p}\left( a_{1},b_{1}\right) }{A_{p}\left( a_{2},b_{2}\right) }
\end{equation*}%
holds if $p\geq 5/3$. It is reversed if \thinspace $p\leq 1$.
\end{corollary}
\end{remark}

\begin{remark}
As sharp upper bounds for the second Seiffert mean, we have the following
relations: 
\begin{equation}
T<\frac{2Q+A}{3}<A_{5/3}<L_{1/3}.  \label{4.2}
\end{equation}%
In fact, it has been shown in \cite[Conclusion 1]{Yang.10(3).2007} that the
function $r\mapsto A_{r}$ is strictly log-concave on $[0,\infty )$, and
therefore 
\begin{equation*}
A_{5/3}^{3/4}A_{1/3}^{1/4}<A_{\frac{3}{4}\frac{5}{3}+\frac{1}{4}\frac{1}{3}%
}=A_{4/3},
\end{equation*}%
which is equivalent with the third inequality in (\ref{4.2}). Now we prove
the second one. Assume that $a>b>0$ and set $\left( a/b\right) ^{1/3}=x\in
\left( 0,1\right) $. Then the inequality in question is equivalent to 
\begin{equation*}
D\left( x\right) :=\ln \frac{2\sqrt{\frac{x^{6}+1}{2}}+\frac{x^{3}+1}{2}}{3}%
-\ln \left( \frac{x^{5}+1}{2}\right) ^{3/5}<0.
\end{equation*}%
Differentiating $D\left( x\right) $ yields%
\begin{equation*}
D^{\prime }\left( x\right) =\frac{3x^{2}\left( 1-x\right) }{\left(
\allowbreak x^{5}+1\right) \left( x^{3}\sqrt{\frac{1}{2}x^{6}+\frac{1}{2}}+%
\sqrt{\frac{1}{2}x^{6}+\frac{1}{2}}+2x^{6}+2\right) }D_{1}\left( x\right) ,
\end{equation*}%
where 
\begin{eqnarray*}
D_{1}\left( x\right)  &=&\left( 1+x\right) \sqrt{\frac{1}{2}x^{6}+\frac{1}{2}%
}-2x^{2}=\frac{\left( \left( 1+x\right) \sqrt{\frac{1}{2}x^{6}+\frac{1}{2}}%
\right) ^{2}-\left( 2x^{2}\right) ^{2}}{\left( 1+x\right) \sqrt{\frac{1}{2}%
x^{6}+\frac{1}{2}}+2x^{2}} \\
&=&\frac{\left( x-1\right) ^{2}\left(
x^{6}+4x^{5}+8x^{4}+12x^{3}+8x^{2}+4x+1\right) }{2\left( 1+x\right) \sqrt{%
\frac{1}{2}x^{6}+\frac{1}{2}}+2x^{2}}>0.
\end{eqnarray*}%
Hence, $D^{\prime }\left( x\right) >0$ for $x\in \left( 0,1\right) $, then $%
D\left( x\right) <D\left( 1\right) =0$.$\allowbreak $
\end{remark}

\begin{remark}
By Theorem \ref{Theorem 1} and \ref{Theorem 2}, the inequalities (\ref%
{Seiffert}) and (\ref{C-S}) can be improved as 
\begin{equation}
N<A_{3/2}<A_{\log _{\pi /2}2}<T<A_{5/3}<A_{2}.  \label{4.1}
\end{equation}%
In our forthcoming paper, we shall establish the sharp bounds for the
Neuman-S\'{a}ndor mean in terms of power means as follows:%
\begin{equation}
A_{p_{0}}<N<A_{4/3},  \label{Y1}
\end{equation}%
where $p_{0}=\frac{\ln 2}{\ln \ln \left( 3+2\sqrt{2}\right) }=\allowbreak
1.\,\allowbreak 222\,8...$.

Thus the chain of inequalities for bivariate means given in \cite[(1)]%
{Costin.IJMMS.2012.inprint} can be refined as a more nice one:%
\begin{equation}
A_{0}<L<A_{1/3}<P<A_{2/3}<I<A_{3/3}<N<A_{4/3}<T<A_{5/3},  \label{Y2}
\end{equation}%
where $L,P,I,N,T$ are the logarithmic mean, the first Seiffert mean,
identric mean, Neuman-S\'{a}ndor mean, the second Seiffert mean,
respectively.
\end{remark}

\end{document}